\documentclass[reqno]{amsart}
\usepackage{}
\usepackage{mathrsfs}
\usepackage{color}
\usepackage{amsmath}
\usepackage{amsfonts}
\usepackage{amssymb}
\usepackage{graphicx}%
\usepackage{hyperref}

 \newtheorem{Theorem}{Theorem}[section]
 \newtheorem{Corollary}[Theorem]{Corollary}
 \newtheorem{Lemma}[Theorem]{Lemma}
 \newtheorem{Proposition}[Theorem]{Proposition}

 \newtheorem{Definition}[Theorem]{Definition}

 \newtheorem{Remark}[Theorem]{Remark}

 \numberwithin{equation}{section}

\begin{document}

\title[Siu's semicontinuity theorem]
 {Lelong numbers, complex singularity exponents, and Siu's semicontinuity theorem}

\author{Qi'an Guan}
\address{Qi'an Guan: School of Mathematical Sciences, and Beijing International Center for Mathematical Research,
Peking University, Beijing 100871, China.}
\email{guanqian@amss.ac.cn}
\author{Xiangyu Zhou}
\address{Xiangyu Zhou: Institute of Mathematics, AMSS, and Hua Loo-Keng Key Laboratory of Mathematics, Chinese Academy of Sciences, Beijing 100190, China}
\email{xyzhou@math.ac.cn}

\thanks{The authors were partially supported by NSFC-11431013. The second author would like to thank NTNU for offering him Onsager Professorship.
The first author was partially supported by NSFC-11522101.}



\date{\today}

\dedicatory{}

\commby{}


\begin{abstract}
In this note, we present a relationship between Lelong numbers and complex singularity exponents.
As an application, we obtain a new proof of Siu's semicontinuity theorem for Lelong numbers.

$\\$
R\'{E}SUM\'{E}. Dans cette note, nous pr\'{e}sentons une relation entre le nombre de Lelong et exposants de singularit\'{e}s complexes.
Comme application, nous obtenons une nouvelle preuve de semicontinuit\'{e} th\'{e}or\`{e}me de Siu pour les num\'{e}ros Lelong.
\end{abstract}

\maketitle
\section{Introduction}
Let $\varphi$ be a plurisubharmonic function near the origin $o\in\mathbb{C}^{n}$.
The Lelong number is defined as follows
\begin{Definition}
\label{def:lelong}
$\nu(\varphi,o):=\sup\{c\geq0:\varphi\leq c\log|z|+O(1)\}.$
\end{Definition}

In \cite{siu74}, Siu established the semicontinuity theorem for Lelong numbers,
i.e.,
the upper level set of Lelong numbers is analytic.
After that,
Kiselman \cite{kisel86} generalized Siu's semicontinuity theorem for directed Lelong numbers.
In \cite{demailly87},
Demailly introduced a generalized Lelong number and generalized the above result of Kiselman.
And then,
Demailly (see \cite{demailly-book}) gave a completely new and simple proof of Siu's theorem
by using Ohsawa-Takegoshi $L^{2}$ extension theorem.

In this note,
we present a new proof of Siu's theorem by
establishing a relationship between Lelong numbers and complex singularity exponents.

Now let us recall the definition of the complex singularity exponent using the concept of multiplier ideal sheaf $\mathcal{I}(\varphi)$
(see \cite{tian87}, see also \cite{demailly-note2000,demailly-book}) (also called log canonical threshold in algebraic geometry see \cite{Sho92,Ko92}):
\begin{Definition}
\label{def:lct}
The complex singularity exponent at $o$ is defined to be
$$c_{o}(\varphi):=\sup\{c\geq0:\mathcal{I}(c\varphi)_{o}=\mathcal{O}_{o}\}.$$
\end{Definition}

Our main result is the following:

\begin{Theorem}
\label{thm:approx_lelong_lct}
Let $\varphi$ be a plurisubharmonic function on $D\subset\subset\mathbb{C}^{n}$ with coordinates $z=(z_{1},\cdots,z_{n})$ containing the origin.
Then for any $k\in\mathbb{N}$,
there exist plurisubharmonic functions $\varphi_{k}$ near the origin of $\mathbb{C}^{n}\times\mathbb{C}^{(2^{k}-1)n}$ with coordinates $(z,w)$
such that

$(1)$ $\varphi_{k}(z,o)=\varphi(z)$;

$(2)$ $\nu(\varphi_{k},(z,o))=\nu(\varphi,z)$;

$(3)$ $\frac{(2^{k}-1)n}{\nu(\varphi,z)}\leq c_{(z,o)}(\varphi_{k})\leq \frac{2^{k}n}{\nu(\varphi,z)}$

for any $z\in D$, where $o$ is the origin in $\mathbb{C}^{(2^{k}-1)n}$.
\end{Theorem}

The second inequality in $(3)$ can be directly deduced by Skoda's result in \cite{skoda72} combined with $(2)$.

\begin{Remark}
It is clear that $(3)$ in Theorem \ref{thm:approx_lelong_lct} is equivalent to
$$\frac{(2^{k}-1)n}{2^{k}n}\nu^{-1}(\varphi,z)\leq c_{(z,o)}((2^{k}n)\varphi_{k})\leq \nu^{-1}(\varphi,z),$$
which implies
$$\{\nu(\varphi,z)\geq \frac{(2^{k}-1)n}{2^{k}n}c\}\supseteq\{c_{(z,o)}((2^{k}n)\varphi_{k})\leq \frac{1}{c}\}\supseteq\{\nu(\varphi,z)\geq c\}.$$

Note that $\lim_{k\to+\infty}\frac{(2^{k}-1)n}{2^{k}n}=1$,
then we obtain
$$\{z|\nu(\varphi,z)\geq c\}=\cap_{k}(\{z|c_{(z,o)}((2^{k}n)\varphi_{k})\leq \frac{1}{c}\}).$$
\end{Remark}

By Berndtsson's solution of openness conjecture \cite{berndtsson13} posed by Demailly and Kollar \cite{D-K01}
(for a proof in two dimensional case see \cite{FM05j,FM05v,FM-book04}),
it follows that
$(\{z|c_{(z,o)}((2^{k}n)\varphi_{k})\leq \frac{1}{c}\})$ is analytic for any $k\in\mathbb{N}$ and $c>0$,
which deduces the following Siu's semicontinuity theorem for Lelong numbers \cite{siu74}
(For Demailly's new proof, the reader is refereed to \cite{demailly-book} and \cite{demailly2010})

\begin{Corollary}
\label{coro:siu}\cite{siu74}
$\{z|\nu(\varphi,z)\geq c\}$ is an analytic set.
\end{Corollary}

We would like to thank the referee for pointing out that Berndtsson's solution of openness conjecture used above to deduce Corollary \ref{coro:siu}
might be replaced by the H\"{o}mander-Bombieri theorem as in Kiselman's proof of Siu's theorem.

\section{Preparation}

\subsection{Restriction formula about complex singularity exponent and Lelong number}

Let $\varphi$ be a plurisubharmonic function on a neighborhood of the origin $o\in\mathbb{C}^{n}$.
In \cite{D-K01}, the following restriction formula ("important monotonicity result") about complex singularity exponents is obtained
by using Ohsawa-Takegoshi $L^2$ extension theorem.

\begin{Proposition}
\label{prop:DK2000}\cite{D-K01}
For any regular complex submanifold $(H,o)\subset(\mathbb{C}^{n},o)$,
\begin{equation}
\label{equ:monotone_lct}
c_{o'}(\varphi|_{H})\leq c_{o}(\varphi)
\end{equation}
holds,
where $\varphi|_{H}\not\equiv-\infty$,
and $o'$ emphasizes that $c_{o'}(u|_{H})$ is computed on the submanifold $H$.
\end{Proposition}

We recall the following restriction property of Lelong numbers.

\begin{Lemma}
\label{lem:lelong}(see \cite{demailly-book})
For any regular complex submanifold $(H,o)\subset(\mathbb{C}^{n},o)$,
\begin{equation}
\label{equ:monotone_lelong}
\nu(\varphi|_{H},o')\geq \nu(\varphi,o)
\end{equation}
holds,
where $\varphi|_{H}\not\equiv-\infty$,
and $o'$ emphasizes that $\nu(\varphi|_{H},o')$ is computed on the submanifold $H$.
\end{Lemma}

\subsection{Lelong number and complex singularity exponent for $U(n)$ invariant plurisubharmonic functions on $\mathbb{C}^{n}$}

We recall the following characterization of
 $U(n)$ invariant plurisubharmonic function (see Lemma III.7.10) in \cite{demailly-book}).

\begin{Lemma}
\label{lem:chara_invar}(see \cite{demailly-book})
Let $\varphi$ be a plurisubharmonic function on $\mathbb{B}^{n}$
which is $U(n)$ invariant.
Then
$\varphi=\chi(\log|z|)$,
where $\chi$ is a convex increasing function.
\end{Lemma}

The folllowing remark is a direct consequence of Lemma \ref{lem:chara_invar}

\begin{Remark}
\label{rem:lct_lelong_u(n)}(see \cite{demailly-book})
Let $\varphi$ be a plurisubharmonic function on $\mathbb{B}^{n}$
which is $U(n)$ invariant.
Then
$c_{o}(\varphi)=\frac{n}{\nu(\varphi,o)}$.
\end{Remark}

\begin{proof}
By Definition \ref{def:lelong},
it is clear that
$$\nu(\varphi,o)=\lim_{t\to-\infty}\frac{\chi(t)}{t},$$
i.e.
for any $\varepsilon>0$,
there exists $\delta>0$
such that for any $|z|<\delta$,
$$(\nu(\varphi,o)+\varepsilon)\log|z|\leq\varphi(z)\leq(\nu(\varphi,o)-\varepsilon)\log|z|,$$
which deduces the present remark by directly calculating.
\end{proof}

\subsection{A holomorphic map and Lelong number}

Define a holomorphic map $p_{m}$ from $\mathbb{C}^{m}\times\mathbb{C}^{m}$ to $\mathbb{C}^{m}$
with coordinates $(z_{1},\cdots,z_{m},w_{1},\cdots,w_{m})$ and $(\tilde{z}_{1},\cdots,\tilde{z}_{n})$
such that
$$p_{m}(z_{1},\cdots,z_{m},w_{1},\cdots,w_{m})=(z_{1}-w_{1},\cdots,z_{m}-w_{m}).$$

Let $\varphi$ be a plurisubharmonic function on $D\subset\mathbb{C}^{m}$,
then $p_{m}^{*}(\varphi)$ is well-defined on $p_{m}^{-1}(D)\subset\mathbb{C}^{m}\times\mathbb{C}^{m}$.

\begin{Lemma}
\label{lem:property}
The construction of $p_{m}^{*}(\varphi)$ implies that for any $z_{0}\in D$ the following statements hold:

$(1)$ $p_{m}^{*}(\varphi)(z_{0},o)=\varphi(z_{0})$;

$(2)$ $\nu(p_{m}^{*}(\varphi),(z_{0},o))=\nu(\varphi,z_{0})=\nu(p_{m}^{*}(\varphi)|_{z=z_{0}},(z_{0},o)')$,
where $(z_{0},o)'$ emphasizes that $\nu(p_{m}^{*}(\varphi)|_{z=z_{0}},(z_{0},o)')$ is computed on the submanifold $\{z=z_{0}\}$.
\end{Lemma}

\begin{proof}
By definition of $p_{m}$,
it follows that $(1)$ and
the second equality of $(2)$ hold.

By $(1)$,
it follows that
$\nu(p_{m}^{*}(\varphi),(z_{0},o))\leq\nu(\varphi,z_{0})$ by Lemma \ref{lem:lelong}.

It suffices to consider the case $z_{0}=(0,\cdots,0)\in D$.
It is clear that
$p_{m}^{*}\log|\tilde{z}|=\log|z-w|\leq\log(|z|+|w|)$.
Then, for any $c>0$ satisfying
$\varphi\leq c\log|\tilde{z}|+O(1)$ $(z\to0)$,
we have
$p_{m}^{*}\varphi\leq c p_{m}^{*}\log|\tilde{z}|+O(1)\leq c\log(|z|+|w|)+O(1),$
which implies
$\nu(p_{m}^{*}(\varphi),(z_{0},o))\geq\nu(\varphi,z_{0})$.

The lemma is proved.
\end{proof}

\subsection{An invariant property of Lelong numbers}

Let $\varphi$ be a plurisuharmonic function on $\Omega\subset\mathbb{C}^{n}\times\mathbb{C}^{m}$ containing the origin,
and let $(z,w)$ denote the coordinates.
One can define
$$\tilde{\varphi}(z,w):=\sup_{g\in U(m)}\varphi(z,gw)$$
on a $U(m)$ invariant neighborhood of $\Omega\cap\{w=o\}$.
It is a plurisubharmonic function.

By definition \ref{def:lelong}, it follows that
\begin{Lemma}
\label{lem:invar}One has
$\nu(\tilde{\varphi},(z_{0},o))=\nu(\varphi,(z_{0},o))$
and
$\nu(\tilde{\varphi}|_{z=z_{0}},(z_{0},o)')=\nu(\varphi|_{z=z_{0}},$ $(z_{0},o)')$
hold for any $(z_{0},o)\in(\Omega\cap\{w=o\})$.
\end{Lemma}

\section{proof of Theorem \ref{thm:approx_lelong_lct}}

Let $m\sim 2^{j}n$ $(j=0,1,\cdots,k-1)$ as in Lemma \ref{lem:property},
where $\sim$ means that the former is replaced by the latter.
We consider the plurisubharmonic function
$p_{2^{k-1}n}^{*}\circ\cdots\circ p_{n}^{*}(\varphi)$ on domain in $\mathbb{C}^{2^{k}n}$.

By Lemma \ref{lem:property},
it follows that

$(1)$ $p_{2^{k-1}n}^{*}\circ\cdots\circ p_{n}^{*}(\varphi)(z,o)=\varphi(z)$;

$(2)$ $\nu(p_{2^{k-1}n}^{*}\circ\cdots\circ p_{n}^{*}(\varphi),(z_{0},o))=\nu(\varphi,z_{0})=\nu(p_{2^{k-1}n}^{*}\circ\cdots\circ p_{n}^{*}(\varphi)|_{\{z=z_{0}\}},(z_{0},o)')$
for any $z_{0}\in D$ and origin $o\in\mathbb{C}^{(2^{k}-1)n}$.

Let
$\varphi_{k}(z,w):=\sup_{g\in U((2^{k}-1)n)}(p_{2^{k-1}n}^{*}\circ\cdots\circ p_{n}^{*}(\varphi)(z,gw))$.

By Lemma \ref{lem:invar} (the first and third equalities) and $(2)$ (the third and fourth equalities),
it follows that
\begin{equation}
\label{equ:0507a}
\begin{split}
\nu({\varphi_{k}}|_{\{z=z_{0}\}},(z_{0},o)')
&=\nu(p_{2^{k-1}n}^{*}\circ\cdots\circ p_{n}^{*}(\varphi)|_{\{z=z_{0}\}},(z_{0},o)')
\\&=\nu(p_{2^{k-1}n}^{*}\circ\cdots\circ p_{n}^{*}(\varphi),(z_{0},o))
\\&=\nu(\varphi_{k},(z_{0},o))
=\nu(\varphi,z_{0}).
\end{split}
\end{equation}

By Remark \ref{rem:lct_lelong_u(n)} ($\mathbb{C}^{n}\sim\{z=z_{0}\}\cap\mathbb{C}^{2^{k}n}(=\{(z_{0},w):w\in\mathbb{C}^{(2^{k}-1)n}\})$, $\varphi\sim{\varphi_{k}}|_{\{z=z_{0}\}}$, $n\sim (2^{k}-1)n$),
it follows that
\begin{equation}
\label{equ:0507b}\nu({\varphi_{k}}|_{\{z=z_{0}\}},(z_{0},o)')=\frac{(2^{k}-1)n}{c_{(z_{0},o)'}({\varphi_{k}}|_{\{z=z_{0}\}})}.
\end{equation}

From Proposition \ref{prop:DK2000}, equality \ref{equ:0507b} and equality \ref{equ:0507a},
it follows that
\begin{equation}
\label{equ:0507c}c_{(z_{0},o)}({\varphi_{k}})\geq c_{(z_{0},o)'}({\varphi_{k}}|_{\{z=z_{0}\}})=\frac{(2^{k}-1)n}{\nu({\varphi_{k}}|_{\{z=z_{0}\}},(z_{0},o)')}
=\frac{(2^{k}-1)n}{\nu(\varphi,z_{0})}.
\end{equation}

Thus we finish proving the present theorem.

\bibliographystyle{references}
\bibliography{xbib}

\end{document}